\documentclass[a4paper, 11pt]{amsart}

\usepackage{latexsym}
\usepackage{amssymb,amsmath}
\usepackage{graphics}
\usepackage{url}
\usepackage[vcentermath, enableskew]{youngtab}

\usepackage{booktabs}  

\newtheorem{theorem}{Theorem}[section]
\newtheorem{corollary}[theorem]{Corollary}

\newtheorem{proposition}[theorem]{Proposition}

\newtheorem{lemma}[theorem]{Lemma}

\theoremstyle{definition}
\newtheorem{definition}[theorem]{Definition}
\newtheorem{remark}[theorem]{Remark}
\newtheorem{example}[theorem]{Example}

\newcommand{\N}{\mathbf{N}}
\newcommand{\Z}{\mathbf{Z}}

\newcommand{\C}{\mathbf{C}}

\renewcommand{\epsilon}{\varepsilon}

\DeclareMathOperator{\sgn}{sgn}

\newcommand{\Ind}{\big\uparrow}
\newcommand{\Res}{\big\downarrow}
\newcommand{\ind}{\!\uparrow}
\newcommand{\res}{\!\downarrow}

\renewcommand{\theta}{\vartheta}

\newcommand{\sh}{\mathrm{sh}}

\newcommand{\mfrac}[2]{{\textstyle\frac{#1}{#2}}}

\newcounter{thmlistcnt}
	{\setcounter{thmlistcnt}{0}%
	\begin{list}{\emph{(\roman{thmlistcnt})}}{%
		\usecounter{thmlistcnt}%
		\setlength{\topsep}{0pt}%
		\setlength{\leftmargin}{0pt}%
		\setlength{\itemsep}{0pt}%
		\setlength{\labelwidth}{17pt}
		\setlength{\itemindent}{30pt}}%
	}%
	{\end{list}}%

\newcommand{\notunrhd}{\hbox{$\hspace*{1pt}\not{\hspace*{-2.5pt}\unrhd}\hspace*{2.25pt}$}}

\newcommand{\tbox}[2]{\draw (#2,-#1)--(1+#2,-#1)--(1+#2,-#1-1)--(#2,-1-#1)--cycle;}
\newcommand{\tlabel}[3]{\node at (#2+0.5,-#1-0.5) {#3};}
\newcommand{\tput}[3]{\node at (#2,-#1){#3};}

\renewcommand{\bar}{\overline}
\newcommand{\size}[1]{\vert #1 \vert}
\usepackage{tikz,mathdots,mathtools}
\usepackage{enumerate}
\newcommand{\height}{\mathrm{ht}}
\newcommand{\ewidetilde}[2]{\widetilde{\smash{#1}\rule{0pt}{#2}}}

\usepackage{kantlipsum} % for text filler

\calclayout

\numberwithin{equation}{section}

\begin{document}
\title[A proof of the Murnaghan--Nakayama rule]{A proof of the Murnaghan--Nakayama rule using Specht modules and tableau combinatorics}
\date{}
\author{Jasdeep Kochhar and Mark Wildon}
\address{Department of Mathematics, Royal Holloway University of London, United Kingdom}
\email{Jasdeep.Kochhar.2015@rhul.ac.uk}
\email{mark.wildon@rhul.ac.uk}
\subjclass[2010]{Primary 20C30. Secondary 05E10, 05E18.}
\keywords{Murnaghan--Nakayama rule, Specht module, polytabloids, traces}
\begin{abstract}
The Murnaghan--Nakayama rule is a combinatorial rule for the character values of symmetric groups.
We give a new combinatorial proof by explicitly  
finding the trace of the representing matrices in the
standard basis of Specht modules. This gives an essentially bijective proof of the rule.
A key lemma is an extension of a straightening
result proved by the second author to skew-tableaux. Our module theoretic methods also give
short proofs of Pieri's rule and Young's rule.
\end{abstract}
\maketitle
\thispagestyle{empty}

\section{Introduction}
In this article we give a new combinatorial proof of the Murnaghan--Nakayama rule
for the values of the ordinary character $\chi^\lambda$ of $S_n$ canonically labelled by the partition $\lambda$
of $n \in \N$. To state the rule, we require the following definitions.

Let $\ell(\lambda)$ denote the number of parts of $\lambda$.
Given partitions $\mu$ and $\lambda$ of $m$ and $m+n$ respectively,
we say that $\mu$ is a {\it subpartition} of $\lambda$, and write $\mu \subseteq \lambda$, 
if $\ell(\mu) \le \ell(\lambda)$ 
and  $\mu_i \le \lambda_i$ for $1 \le i \le \ell(\mu)$. 
We define the \emph{skew diagram} $[\lambda/\mu]$ to be the set of \emph{boxes}
\[ \{(i,j) : 1\le i \le t \mbox{ and } \mu_i < j \le \lambda_i \}, \]
and call $\lambda/\mu$ a \emph{skew partition}.
%When $\mu$ is empty we omit it from the notation $\lambda /\mu$.
We define \emph{row $k$} (resp.~\emph{column $k$}) of $\lambda/\mu$ to be the subset of $[\lambda/\mu]$ of boxes whose first (resp.~second) coordinate equals~$k$. Let
$\height(\lambda/\mu)$ be one less than the number of non-empty rows of $[\lambda/\mu]$.
We define a \emph{border strip}
to be a skew partition whose skew diagram is connected and which
contains no four boxes forming the partition $(2,2)$.

\begin{theorem}[{Murnaghan--Nakayama rule}]\label{thm: MN rule}
Given $m,n \in \N,$ 
let  $\rho \in S_{m+n}$ be an $n$-cycle and let 
$\pi$ be a permutation of the remaining $m$ numbers. 
Then 
\[ \chi^\lambda(\pi\rho) = \sum (-1)^{\height(\lambda/\mu)} \chi^\mu (\pi),\] 
where the sum is over all $\mu \subset \lambda$ such that $\size{\mu} = m$ and $\lambda/\mu$ is a 
border strip.
\end{theorem}

Before we continue we provide an example of the Murnaghan--Nakayama rule, showing how it can be applied recursively to calculate single character values. 

\begin{example}\label{ex: MNRule}
Let $\sigma = (1,2)(3,4,5,6,7)(8,9,10,11,12) \in S_{12}$.
We evaluate\\ $\chi^{(4,4,4)}(\sigma)$.
Taking $\rho = (8,9,10,11,12)$, we begin by removing border strips of size 5 
from $(4,4,4)$. As shown in Figure~\ref{fig: border strips} there are two such strips,
namely $(4,4,4)/(4,3)$ and $(4,4,4) / (3,3,1)$, of heights $1$ and $2$, respectively. 
Therefore by the Murnaghan--Nakayama rule
\[\chi^{(4,4,4)}(\sigma) = (-\chi^{(4,3)} + \chi^{(3,3,1)})\bigl((1,2)(3,4,5,6,7) \bigr).\]
Two further applications of the Murnaghan--Nakayama rule to each summand now show
that 
%Taking $\rho = (3,4,5,6,7)$ we remove
%border strips of size 5 
%indicated by the black lines in Figure \ref{fig: border strips}.
%The first of these has height 0, and the second has height 2. 
%Therefore 
$\chi^{(4,4,4)}(\sigma) = (\chi^{(2)} + \chi^{(2)})\bigl( (1,2) \bigr) = 1 + 1 = 2$.
%By Theorem~\ref{thm: filtration} and three applications of Proposition~\ref{prop: unique},
%the two standard tableaux $s$ of shape $(4,4,4)$ such that the polytabloid $e(s)$
%has a non-zero coefficient in the expression of $e(s \sigma)$ as an integral
%linear combination of standard polytabloids, \emph{and} $\{1,2\}$, $\{3,4,5,6,7\}$, $\{8,9,10,11,12\}$
%occupy the boxes corresponding to the border strips shown in Figure~\ref{fig: border strips}, are
%   \begin{marginpar}{\hskip0.4in$\young(:::1,:::2,:345)$}
%\end{marginpar}
%\begin{marginpar}{\smallskip\hskip0.4in$\young(::1,234,5)$}
%\end{marginpar}
%\begin{marginpar}{\smallskip\hskip0.4in$\young(12)$}
%\end{marginpar}
%\[ \young(1237,4568,9\ten\eleven\twelve)\; , \quad
%   \young(1238,4569,7\ten\eleven\twelve)\; . \]  
%respectively. In both cases the coefficient is $1$. 
%For example, the second tableau above is constructed from
%the skew-tableaux $s_{(4,4,4)/(3,3,1)}$, $s_{(3,3,1)/(2)}$ and $s_{(2)}$ defined
%in Definition~\ref{defn: columnar} %, and shown in the margin, 
%by adding $7$, $2$ and $0$, respectively,
%to each of their entries. 
%Dropping condition on boxes there are four more tableaux, which overall give zero
%contribution to the trace:
% 1378 \\ 249J \\ 56TQ (+1), 1378 \\ 249T \\ 56JQ (-1) ; paired on same filtration
% 1356 \\ 2468 \\ 9TJQ (+1), 1347 \\ 2568 \\ 9TJQ (-1) ; paired on same filtration

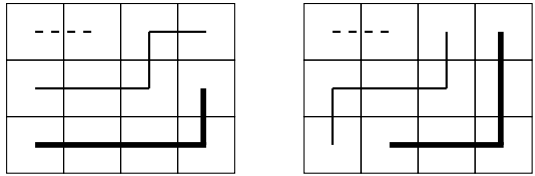
\begin{figure}
\begin{tikzpicture}[scale = 0.75]
\tbox{0}{0} \tbox{0}{1} \tbox{0}{2} \tbox{0}{3}
\tbox{1}{0}\tbox{1}{1} \tbox{1}{2} \tbox{1}{3}
\tbox{2}{0}\tbox{2}{1} \tbox{2}{2} \tbox{2}{3}
\draw[black, line width = .75mm](0.5,-2.5) -- (3.5,-2.5);
\draw[black, line width = .75mm](3.45,-2.5) -- (3.45,-1.5);
\draw[thick](0.5,-1.5) -- (2.5,-1.5);
\draw[thick](2.5,-1.5) -- (2.5,-0.5);
\draw[thick](2.5,-0.5) -- (3.5,-0.5);
\draw[thick, dashed](0.5,-0.5)--(1.5,-0.5);

\end{tikzpicture}
\qquad
%\;\raisebox{0.75cm}{,} \quad 
\begin{tikzpicture}[scale = 0.75]
\tbox{0}{0} \tbox{0}{1} \tbox{0}{2} \tbox{0}{3}
\tbox{1}{0}\tbox{1}{1} \tbox{1}{2} \tbox{1}{3}
\tbox{2}{0}\tbox{2}{1} \tbox{2}{2} \tbox{2}{3}
\draw[black, line width = .75mm](1.5,-2.5) -- (3.5,-2.5);
\draw[black, line width = .75mm](3.45,-2.5) -- (3.45,-0.5);
\draw[thick](0.5,-1.5) -- (2.5,-1.5);
\draw[thick](2.5,-1.5) -- (2.5,-0.5);
\draw[thick](0.5,-2.5) -- (0.5,-1.5);
\draw[thick, dashed](0.5,-0.5)--(1.5,-0.5);
\end{tikzpicture}
\caption{The border strips of size $5$ (solid) and $2$ (dashed)
removed to compute $\chi^{(4,4,4)}(\sigma)$ in Example \ref{ex: MNRule}. }\label{fig: border strips}
\end{figure}
\end{example}
As Stanley notes in \cite[page 401]{StanleyEnum2}, the Murnaghan--Nakayama rule
was first proved by Littlewood and Richardson in \cite[\S 11]{LR}. Their proof derives
it, essentially as stated in Theorem~\ref{thm: MN rule}, as a corollary of the older Frobenius
formula \cite[page 519, (6)]{Frobenius1900} for the characters of symmetric groups.
(For a modern statement of the Frobenius formula see \cite[(7.77)]{StanleyEnum2} or \cite[(4.10)]{FH}.)
Later Murnaghan 
\cite[page 462, (13)]{Murnaghan1937} gave
a similar but independent derivation of the rule. Murnaghan's paper was cited by Nakayama 
\cite[page 183]{Nakayama1941}, who
gave a more concise proof, still from the Frobenius formula. James gave a different proof in \cite[Ch.~11]{J}
using the relatively deep  Littlewood--Richardson rule.
More recently, elegant involutive proofs have been given by Mendes and Remmel \cite[Theorem~6.3]{MR}
using Pieri's rule and Young's rule and by Loehr 
\cite[\textsection 11]{L} using his labelled abacus representation of antisymmetric functions. 

%\subsection{Outline of the proof}\label{sec: proofOutline}
The starting point for our proof is Corollary~\ref{cor: reduction} of Theorem~\ref{thm: filtration} below, 
which states that
$\chi^\lambda(\pi \rho) = \sum_\mu \chi^\mu(\pi) \chi^{\lambda/\mu}(\rho)$,
where $\chi^{\lambda / \mu}$ is the ordinary character of the skew Specht module $S^{\lambda / \mu}$ 
defined in~\S\ref{sec: skew_back}. By this corollary, it suffices
to show that if $\rho$ is an $n$-cycle then
\begin{equation}\label{eq: skewCase}
\chi^{\lambda / \mu}(\rho) = \begin{cases} (-1)^{\height(\lambda/\mu)}& \text{if 
$\lambda/\mu$ is a border strip of size $n$} \\
0 & \text{otherwise.} \end{cases} \end{equation}
We do this by  explicitly computing the trace of the 
matrix representing the $n$-cycle $\rho$ in the standard basis (see Theorem~\ref{thm: SBT}) 
of $S^{\lambda/\mu}$.
In the critical case where $\lambda/\mu$ is a border strip, we show that there is a unique
basis element giving a non-zero contribution to the trace. This gives a new and essentially
bijective proof of the Murnaghan--Nakayama rule.

%\subsection{Background to Theorem~\ref{thm: filtration}}
Theorem~\ref{thm: filtration} is the main result in~\cite{JP}. 
The proof in \cite{JP} constructs skew Specht modules as ideals in the group algebra of $S_n$
over a field. Our proof using polytabloids instead generalizes James' proof of the modular
branching rule for Specht modules \cite[Ch.~9]{J}. 
In this way we obtain 
a stronger isomorphism for integral modules
that replaces the lexicographic order used in \cite{J} and \cite{JP} with the dominance order.

%\subsection{Outline}
In \textsection \ref{sec: skew_back} and \textsection \ref{sec: Garnir} we 
%define $\lambda/\mu$-tableaux and 
define $\lambda/\mu$-polytabloids and
state Theorem \ref{thm: SBT}, which says that the set of standard 
$\lambda/\mu$-polytabloids is a $\Z$-basis of $S^{\lambda/\mu}$. 
In \textsection \ref{sec: filtration} we prove Theorem~\ref{thm: filtration}
and deduce Corollary~\ref{cor: reduction}. 
In  \textsection \ref{sec: Pieri} we use Theorem~\ref{thm: filtration} to 
give short module-theoretic proofs of Pieri's rule and Young's rule.
In \textsection \ref{sec: dominance} we prove Lemma \ref{lemma: dominance}, which
gives a necessary condition for a standard polytabloid to
appear with a non-zero coefficient when a given $\lambda/\mu$-polytabloid 
is written as a linear combination of standard polytabloids. This 
%We emphasise the use of Lemma \ref{lemma: dominance} as it 
generalises Proposition 4.1 in \cite{W} to skew tableaux.
In \textsection \ref{sec: snakeCase} we use Lemma~\ref{lemma: dominance}
to give a bijective proof of~\eqref{eq: skewCase} when $\lambda/\mu$ is a border strip.
We then deal with the remaining case in \textsection \ref{sec: endgame} by 
a short argument using Pieri's rule and Young's rule.

\section{Background}
\subsection{Skew tableaux and skew Specht modules}\label{sec: skew_back}
Fix $m$, $n \in \N$. Let $\lambda$ be a partition of $m  + n$ and let $\mu$ be a subpartition of
$\lambda$ of size $m$.
We define a \emph{$\lambda/\mu$-tableau} $t$ to be a bijective function $t : [\lambda/\mu] \rightarrow \{1,2,\ldots, n\}$, and call $t$ a \emph{skew tableau} of \emph{shape} $\lambda/\mu$.
We call $(i,j)t$ the \emph{entry of $t$} in position $(i,j)$.
Thus a $\lambda/\mu$-tableau can be visualized as a
filling of the boxes $[\lambda/\mu]$ with distinct entries from $\{1,\ldots, n\}$. 
We draw skew diagrams using the `English convention' in which the largest part appears at the top of the page: thus the \emph{top row}
is row $1$, and so on. 
The \emph{conjugate} partition of $\lambda$ is the partition
$\lambda'$ whose diagram $[\lambda']$ is obtained by reflecting $[\lambda]$ in its
leading diagonal. Equivalently, $\lambda^\prime_j = | i : \lambda_i \ge j|$.

There is a natural action of $S_n$ on the set of $\lambda/\mu$-tableaux defined
by $(i,j)(t \sigma) = \bigl( (i,j)t \bigr)\sigma$ for $\sigma \in S_n$.
Given a $\lambda/\mu$-tableau $t,$ let $R(t)$ (resp.~$C(t)$) be the subgroup of $S_n$ consisting of all permutations that setwise fix the entries in each row (resp.~column) of $t$. 
We define an equivalence relation $\backsim$ on the set of $\lambda/\mu$-tableaux by $t \backsim u$ if and only if there exists $\pi \in R(t)$ such that $u =  t\pi.$
The {\it $\lambda/\mu$-tabloid} $\{t\}$ is the equivalence class of $t$. A short calculation
shows that $S_n$ acts
on the set of $\lambda/\mu$-tabloids by $\{t\} \sigma
= \{t \sigma\}$. % for $\sigma \in S_n$.
 
Generalizing the usual definitions to skew partitions, 
we say that a $\lambda/\mu$-tableau is \emph{row standard} if the 
entries in its rows are increasing when read from left to right, and \emph{column standard}
if the entries in its columns are increasing when read from top to bottom.
A tableau $t$ that is both row standard and column standard is a {\it standard} tableau.
 
Let $M^{\lambda/\mu}$ be the $\Z S_n$-permutation module
spanned by the $\lambda/\mu$-tabloids.
We define the \emph{$\lambda/\mu$-polytabloid}  $e(t) \in M^{\lambda/\mu}$ by
\vspace*{-2.5pt}
\[ e(t) = \sum_{\sigma\in C(t)} \text{sgn}(\sigma)\{t\}\sigma. \]
If $t$ is a standard tableau then we say that $e(t)$ is a \emph{standard polytabloid}.
The \emph{skew Specht module} $S^{\lambda/\mu}$ is then the $\Z S_n$-module 
spanned by all $\lambda/\mu$-polytabloids. Taking $\mu = \varnothing$ this is
the Specht module $S^\lambda$, defined over~$\Z$. By definition,
$\chi^\lambda$ is the character of $S^\lambda \otimes_\Z \C$, and more
generally, $\chi^{\lambda / \mu}$ is the character of $S^{\lambda / \mu} \otimes_\Z \C$.
%
%Finally, we define the 
%\emph{conjugate} partition $\lambda^\prime$ of $\lambda$ by
%\[\lambda^\prime_j = | i : \lambda_i \ge j|.\]
%The diagram $[\lambda']$ is obtained from $[\lambda]$ by reflecting it in its leading diagonal.

\subsection{Garnir relations and the Standard Basis Theorem}\label{sec: Garnir}
If $\sigma \in S_n$ then 
an easy calculation shows that 
\vspace*{-5pt}
\begin{equation}
\label{eq:cyclic} 
e(t)\sigma = e(t\sigma).
\end{equation}
 Hence $S^{\lambda/\mu}$ is cyclic,
generated by any $\lambda/\mu$-polytabloid. Moreover given $\tau \in C(t)$ then
\vspace*{-5pt}
\begin{equation}\label{eq: column} e(t)\tau = \sgn(\tau)e(t) \end{equation}
so $S^{\lambda/\mu}$ is spanned by the $\lambda/\mu$-polytabloids $e(t)$ for $t$ a column standard $\lambda/\mu$-tableau. 
Let $\widetilde{t}$ be the unique
column standard $\lambda/\mu$-tableau whose columns agree setwise with $t$
and let $\epsilon_t \in \{+1,-1\}$ be defined by % the sign of the column straightening permutation, so 
$e(\,\widetilde{t}\,) = \epsilon_t\hskip0.5pt e(t)$. We call $\widetilde{t}$ the
\emph{column straightening} of $t$.

%We shall make use of various relations in $S^{\lambda/\mu}.$
%In particular, 
%Garnir relations (see \cite[\textsection 7]{J}) have analogues for skew-Specht modules.
Suppose that $(i,j)$ and $(i,j+1)$ are boxes in $[\lambda/\mu].$
Given a $\lambda/\mu$-tableau~$t$, let 
\[ X = \{(i,j)t, (i+1,j)t, \ldots \} \]
be the set of entries in column $j$ of $t$ 
weakly below box $(i,j)$, and let~
\[ Y = \{\ldots, (i-1,j+1)t, (i,j+1)t\} \]
be the set of entries in column $j+1$ of $t$ weakly above box 
$(i,j+1)$. Let $C_{X,Y}$ be the set of all products of transpositions $(x_1,y_1) \ldots (x_k, y_k)$
for $x_1 < \ldots < x_k$ and $y_1 < \ldots < y_k$ where $\{x_1,\ldots,x_k\} \subseteq X$ and
$\{y_1,\ldots, y_k\} \subseteq Y$ are non-empty $k$-sets. 
We define the \emph{Garnir element for $X$ and $Y$} by
\begin{equation}
\label{eq: GarnirElement} 
G_{X,Y} = 1 + \sum_{\sigma \in C_{X,Y}} \sgn(\sigma)\sigma \in \Z S_{X \cup Y}.
\end{equation}
Restated, replacing ideals in the group ring $\Z S_n$ with polytabloids, (3.8) in~\cite{FP}
implies that
\vspace*{-5pt}
\begin{equation}\label{eq: Garnir} e(t) G_{X,Y} = 0.\end{equation}
%We remark that these column relations and Garnir relations generate the full annihilator of $e_t$ in $\Z S_n.$ 
%In \cite[Theorem 3.9]{FP} the authors
%use the column straightening relations~\eqref{eq: column}
%and the Garnir relations~\eqref{eq: Garnir}
%to prove the following result. 
\vspace*{-2.5pt}
Similarly restated, Theorem 3.9 in \cite{FP} is as follows.

\begin{theorem}[Standard Basis Theorem]{\ }\label{thm: SBT}
\begin{itemize}
\item[(i)] 
Any $\lambda/\mu$-polytabloid can be expressed as a $\Z$-linear combination of standard $\lambda/\mu$-polytabloids
by applications of column relations~\eqref{eq: column} and Garnir relations~\eqref{eq: Garnir}.
\item[(ii)] 
The $\Z S_n$-module $S^{\lambda/\mu}$ has  the set of standard $\lambda/\mu$-polytabloids
as a $\Z$-basis.
\end{itemize}
\end{theorem}

We remark that the proofs of Theorem 7.2 and 8.4 in \cite{J}, for Specht modules labelled by partitions,
but defined using polytabloids, generalize easily to prove~\eqref{eq: Garnir} and Theorem~\ref{thm: SBT}
exactly as stated above.
We give a small example of Garnir relations in Example~\ref{ex: Garnir} below.

\subsection{A filtration for Specht modules}\label{sec: filtration}
We  require the following notation. 
Given finite groups $G$ and $H,$ a $\Z G$-module $U$ and a $\Z H$-module $V$, we
denote by $U \boxtimes V$  the $\Z [G\times H]$-module given by the outer tensor product (see \cite[(43.1)]{CR}) of $U$ and $V$. 
The induction and restriction of modules and 
characters, defined as in \cite[\S 12D, \S 43]{CR}, 
are denoted by $\ind$ and $\res$, respectively. 

Fix throughout this section $m$, $n \in \N$ and a partition $\lambda$ of $m+n$.
Let $S_{(m,n)}\hspace{-2pt} =  S_{\{1,2,\ldots,m\}} \times S_{\{m+1,m+2,\ldots,m+n\}}$.
We shall prove the following theorem. 

\begin{theorem}[{\cite[Theorem 3.1]{JP}}]\label{thm: filtration}
The restricted Specht module $S^{\lambda}\hskip-0.5pt\res_{S_{(m,n)}}$ has a descending chain of $\Z S_{(m,n)}$-submodules whose successive quotients are isomorphic~to
$S^{\mu} \boxtimes S^{\lambda/\mu}$, 
where each subpartition~$\mu$ of $\lambda$ of size $m$ occurs exactly once.
\end{theorem}

%Theorem~\ref{thm: filtration} may be thought of as a generalization of James' modular branching
%rule for Specht modules \cite[Ch.~9]{J}, replacing $S_{n-1}$ with $S_{(m,n)}$. Our proof requires
%these preliminaries.
%The following preliminaries are required.

%We start by defining the modules in this filtration.
Suppose that $\lambda$ has first part $c$.
Given a $\lambda$-tableau $t$ we define the \emph{$m$-shape} of $t$ to be
the composition $(\gamma_1,\ldots, \gamma_c)$ such that
$\gamma_j$ equals the number of entries in column $j$ of $t$ not exceeding $m.$
Let $\unrhd$ denote the dominance order on compositions of the same size, defined by $\delta \unrhd \gamma$
if and only if $\ell(\delta) \le \ell(\gamma)$ and $\sum_{i=1}^k \delta_i \ge \sum_{i=1}^k \gamma_i$
whenever $1 \le k \le \ell(\delta)$.
For each composition $\gamma$ such that $\ell(\gamma) \le c$ 
we define
\[ V^{\unrhd \gamma} = \langle e(t) : \text{$t$ a column standard $\lambda$-tableau of $m$-shape $\delta$
where $\delta \unrhd \gamma$} \rangle_\Z. \]
Note that the definition of the $m$-shape agrees with the notation $b(y)$ in the proof of \cite[Theorem 3.1]{JP}.
%and $V^{\unrhd \gamma} = \sum_{\delta \unrhd \gamma} U^\delta$.
%We shall show that the $V^{\unrhd \mu'}$, where $\mu$ is a subpartition of $\lambda$ of size $m$,
%give a filtration of $S^{\lambda}$ satisfying
% Theorem~\ref{thm: filtration}. (Here $\mu'$ denotes the conjugate
%partition to~$\mu$.) 
We require the following total ordering on the set of column standard $\lambda$-tableaux, 
defined implicitly in \cite[page 30]{J}. 

\begin{definition}\label{defn: totalOrder}
Let $u$ and $t$ be column standard $\lambda$-tableaux.
We write $u > t$ if and only if the greatest entry appearing in a different column
in $u$ to $t$ appears further right in $u$ than $t$.
%there exists an entry $x$ such that 
%\begin{itemize}
%\item[(i)] for all $y > x,$ the column containing $y$ is the same in $t$ and $u$
%\item[(ii)] the column containing $x$ in $u$ is to the right of the column containing $x$ in $t$.
%\end{itemize}
\end{definition}

For instance, the $>$ order
on column standard $(2,2)$-tableaux is
\[ \young(13,24) > \young(12,34) > \young(21,34) > \young(12,43) > \young(21,43) > \young(31,42)\hspace{0.5pt}. \]
Note that here, as in general, the greatest tableau under $>$ is standard.
Several times below we use
that if $x > y$ and $x$ is to the left of $y$ in the column standard tableau $u$
then  $\ewidetilde{u(x,y)}{6pt} > u$. 

\begin{proposition}\label{prop: colStdToStd}
Let $u$ be a column standard $\lambda$-tableau of $m$-shape $\gamma$. Then
$e(u)$ is equal to a $\Z$-linear combination of standard $\lambda$-polytabloids $e(t)$ where each
$t$ has $m$-shape $\mu'$ for some partition $\mu$ such that $\mu' \unrhd \gamma$.
\end{proposition}

\begin{proof}
If $u$ is standard then $\gamma$ is a partition, and there is nothing to prove.
If~$u$ is not standard then there exists $(i,j) \in [\lambda]$ such that $(i,j)u > (i,j+1)u$.
Let $X$ and $Y$ be as defined in~\eqref{eq: GarnirElement}.
% = \{(i,j)u, (i+1,j)u, \ldots \}$ be the entries of $u$ in column $j$ weakly below $(i,j)$
%and let $Y = \{\ldots, (i-1,j+1)u, (i,j+1)u\}$ be the entries of $u$ in column $j$ weakly above $(i,j+1)$.
By~\eqref{eq: Garnir} we have
\[ 0 = e(u) + \sum_{\sigma \in C_{X,Y}} \epsilon_{u \sigma}\sgn(\sigma)  e(\widetilde{u \sigma}) \]
where $\widetilde{u \sigma}$ %is the column straightening of $u\sigma$ 
and
$\epsilon_{u \sigma} \in \{+1,-1\}$ are %is 
as defined 
at the start of \S\ref{sec: Garnir}.
Let $\sigma \in C_{X,Y}$.
Since the minimum of $X$ exceeds the maximum of $Y$, 
we have $x > y$ for each transposition $(x,y)$ in~$\sigma$.
Hence $\widetilde{u \sigma} > u$. 
Write~$\delta$ for the $m$-shape of $\widetilde{u \sigma}.$ 
If there are exactly $k$ transpositions $(x,y)$ in $\sigma$
such that $x > m \ge y,$ then $\delta_j = \gamma_j + k$, $\delta_{j+1} = \gamma_{j+1}-k$
and $\delta_{j'} = \gamma_j$ for $j' \not= j, j+1$. Hence $\delta \unrhd \gamma$.
The lemma now follows by induction on the $\ge$ and $\unrhd$ orders.
\end{proof}

%\begin{corollary}\label{cor: SBTmu}
%Let $\mu$ be a subpartition of $\lambda$. Then $V^{\unrhd \mu'}$ is
%a $\Z$-submodule of $S^\lambda$ with $\Z$-basis the
%standard $\lambda$-tableau of $m$-shape $\nu'$ where $\nu' \unrhd \mu'$.
%\end{corollary}
%
%\begin{proof}
%Since the standard $\lambda$-tableau are linearly independent by Theorem~\ref{thm: SBT}, this follows immediately
%from Proposition~\ref{prop: colStdToStd}.
%\end{proof}

\begin{corollary}\label{cor: SBTmu}
Let $\mu$ be a subpartition of $\lambda$ of size $m$. Then $V^{\unrhd \mu'}$ is
a $\Z S_{(m,n)}$-submodule of $S^\lambda$ with $\Z$-basis given by the
standard $\lambda$-tableaux of $m$-shape $\nu'$ such that $\nu' \unrhd \mu'$.
\end{corollary}

\begin{proof}
Since the standard $\lambda$-polytabloids are linearly independent by Theorem~\ref{thm: SBT}(ii), 
it follows immediately
from Proposition~\ref{prop: colStdToStd} that $V^{\unrhd \mu'}$ has a $\Z$-basis as claimed.
If $\pi \in S_{(m,n)}$ and $s$ is a standard $\lambda$-tableau of $m$-shape $\nu'$
then $s \pi$ also has $m$-shape $\nu'$, as does $\widetilde{s \pi}$. 
By~\eqref{eq: column} and Proposition~\ref{prop: colStdToStd}, 
$e(s\pi) = \pm e(\widetilde{s\pi}) \in V^{\unrhd \nu'} \subseteq V^{\unrhd \mu'}$. Hence $V^{\unrhd \mu'}$ is a
$\Z S_{(m,n)}$-module.
\end{proof}

%We now refine the argument in Proposition~\ref{prop: colStdToStd}
%to prove a precise result on the action of $S^{(m,n)}$ on $V^{\unrhd \mu'}$.
Given a  $\mu$-tableau $u$ with (as usual) entries  $\{1,\ldots, m\}$
and a $\lambda / \mu$-tableau~$v$ with entries $\{m+1,\ldots, m+n\}$, let
$u \hskip0.5pt\cup\hskip0.5pt v$ denote the  $\lambda$-tableau defined by 
\[ (i,j)(u \cup v) = \begin{cases} (i,j)u & \text{if $(i,j) \in [\mu]$} \\
(i,j)v & \text{if $(i,j) \in [\lambda/\mu]$.} \end{cases}\]
Clearly every $\lambda$-tableau of $m$-shape $\mu'$ is of this form. 
We shall show
that the action of $S_{(m,n)}$ on standard $\lambda$-polytabloids is compatible
 with this factorization. We require the following lemma and proposition,
 which are illustrated in Example~\ref{ex: Garnir} below.

\begin{lemma}\label{lemma: GarnirSplit}
Let $\mu$ be a subpartition of $\lambda$ of size $m$.
Let $u$ be a column standard
$\mu$-tableau and let $v$ be a $\lambda/\mu$-tableau.
Let $(i,j) \in [\mu]$ be a box such
that 
\[ m \ge (i,j)u > (i,j+1)u.\] 
Let $r = \mu_j'$ so $(r, j)$ is the lowest box in column $j$ of $u$, and define
\begin{align*}
X &= \{(i,j)u, (i+1,j)u, \ldots, (r,j)u, (r+1,j)v, \ldots \}, \\
Y &= \{\ldots, (i-1,j+1)u, (i,j+1)u \}, \\
X^\star &= \{(i,j)u,(i+1,j)u, \ldots, (r,j)u\}.
\end{align*}
Let $C_{X^\star,Y} = \{\sigma \in C_{X,Y} : x \sigma = x \text{ for all $x \in X \backslash X^\star$} \}$.
Then
\[ 0 = e(u \cup v) +  \sum_{\sigma^\star \in C_{X^\star,Y}} \sgn(\sigma^\star)e(u \cup v) \sigma^\star
 + 
  \sum_{\sigma \in C_{X,Y} \backslash C_{X^\star,Y}} \sgn(\sigma)e(u \cup v)\sigma \]
where
%, for each $\sigma$, $e(u \cup v) \sigma$ is a $\Z$-linear combination of standard polytabloids
%$e(s)$  for tableaux $s$ of $m$-shape $\nu'$ where $\nu' \rhd \mu'$.
\begin{itemize}
\item[(i)] for each $\sigma^\star$, we have $e(u \cup v)\sigma^\star = e(u \sigma^\star \cup v)$
and $\widetilde{u\sigma^\star} > u$;
%$e(u \cup v) \sigma^\star = e(u\sigma^\star \cup v)$;
\item[(ii)] for each $\sigma$, $e(u \cup v) \sigma$ is a $\Z$-linear combination of polytabloids
$e(s)$ for standard tableaux $s$ of $m$-shape $\nu'$ where $\nu' \rhd \mu'$.
\end{itemize}
\end{lemma}

\begin{proof}
Since $G_{X,Y} = 1 + \sum_{\sigma^\star \in
C_{X^\star,Y}} \sgn(\sigma^\star)\sigma^\star 
+ \sum_{\sigma \in C_{X,Y} \backslash C_{X^\star,Y}} \sgn(\sigma)\sigma$,
the displayed equation follows from~\eqref{eq: Garnir}.
%, noting that 
%$e(u \cup v)\sigma^\star = e(u \sigma^\star \cup v)$.
Since $C_{X^\star, Y} \subseteq S_{\{1,\ldots, m\}}$,
 (i) follows from the observation
after Definition~\ref{defn: totalOrder}.
%, since $\sigma^\star \in S_{\{1,\ldots, m\}}$
%for each $\sigma^\star$. 
Take $\sigma \in C_{X , Y} \backslash C_{X^\star , Y}$ and let $w = (u \cup v) \sigma$.
Since $\sigma$ involves a transposition
$(x,y)$ with $x > m \ge y$, the statistic $k$ in the proof of Proposition~\ref{prop: colStdToStd}
is non-zero. Hence the $m$-shape of $e(\widetilde{\hskip-0.5ptw})$ is $\delta$
for some composition
$\delta$ with $\delta \rhd \mu'$. The statement
of Proposition~\ref{prop: colStdToStd} now implies that 
$e(\widetilde{\hskip-0.5ptw})$ 
is a $\Z$-linear combination of standard polytabloids $e(s)$ for~$s$ of $m$-shape $\nu'$ where
$\nu' \unrhd \delta$. Hence $\nu' \rhd \mu'$, as required for (ii).
\end{proof}

\begin{proposition}\label{prop: GarnirSplit}
Let $\mu$ be a subpartition of $\lambda$ of size $m$. Let $u$ be a column standard
$\mu$-tableau and let $t$ be a standard $\lambda/\mu$-tableau.
If $e(u) = \sum_S \alpha_S e(S)$
where the sum is over all standard $\mu$-tableaux $S$ and $\alpha_S \in \Z$ for each $S$ then
\[ e(u \cup t) \in \sum_S \alpha_s e(S \cup t) + \sum_{\nu' \rhd \mu'} V^{\unrhd \nu'}. \]
\end{proposition}

\begin{proof}
If $u$ is standard the result is obvious. 
If not, there exists a box $(i,j) \in [\mu]$ such
that $m \ge (i,j)u > (i+1,j)u$. Let $X^\star$ and $Y$ be as in~Lemma~\ref{lemma: GarnirSplit}.
%By~\eqref{eq: Garnir} we have
%%$e(u) = - \sum_{\sigma \in C_{X^\star,Y}} \sgn(\sigma^\star) e(u) \sigma^\star$. Thus
%\[ e(u) \otimes e(t)  
%= - \sum_{\sigma^\star \in C_{X^\star,Y}} \sgn(\sigma^\star) 
%%\bigl(
% e(u \sigma^\star) \otimes e(t). %\bigr). 
% \]
By Lemma~\ref{lemma: GarnirSplit}(ii) we have
\[ e(u \cup t)  \in -\sum_{\sigma^\star \in C_{X^\star, Y}} \sgn(\sigma^\star) e(u  \cup t) \sigma^\star
+ \sum_{\nu' \rhd \mu'} V^{\unrhd \nu'}. \]
Using Lemma~\ref{lemma: GarnirSplit}(i), the result now follows by induction on the $\ge$ order.
\end{proof}

We also need the analogous lemma in which $u$ is a $\lambda/\mu$-tableau, $(i,j) \in [\lambda/\mu]$ and $(i,j)u > (i,j+1)u > m$, and
$Y^\star = \{(r,j+1)u,\ldots, (i,j+1)u\}$ where now $r = \mu_{j+1}'+1$,
%where $r$ is minimal such that $u(r,j+1) > m$,
and the relevant sets of coset representatives are $C_{X,Y^\star}$ and $C_{X,Y} \backslash C_{X,Y^\star}$.
It implies the~analogous proposition in which $e(t \cup v)$ is straightened, where $t$ is a standard
$\mu$-tableau and $v$ is a column standard $\lambda/\mu$-tableau.
The proofs are entirely analogous.

The following example makes explicit the statements of Lemma \ref{lemma: GarnirSplit} and Proposition \ref{prop: GarnirSplit}. 

\begin{example}\label{ex: Garnir}
Let  $u$, $t$ and $u \cup t$ be the skew tableaux shown below.
\[u =\, \young(12,43)\,, \quad 
t=\, \young(::5,::7,68)\,, \quad u \cup t = \,
\young(125,437,68)\, .\]
As $4 = (2,1)(u\cup t) > (2,2)(u\cup t) = 3$, to straighten $u \cup t$ we define $X = \{4,6\}$ and $Y = \{2,3\}.$
The relation $e(u \cup t)G_{X,Y} = 0$ gives
\begin{align*}
e(u \cup t) = &-e\!\left(\,\young(135,247,68)\,\right) 
+ e\!\left(\,\young(125,347,68)\,\right)%\label{eq: XstarY}
\\
&+e\!\left(\,\young(135,267,48)\,\right) 
- e\!\left(\,\young(125,367,48)\,\right) - e\!\left(\,\young(145,267,38)\,\right).%\label{eq: notXstarY}.
\end{align*}
In the notation of Lemma \ref{lemma: GarnirSplit}, we have $X^\star = \{4\}.$ 
The standard polytabloids in the top and bottom lines 
come from the permutations 
in $C_{X^\star,Y}$ and $C_{X,Y} \backslash C_{X^\star,Y}$, respectively.
Furthermore, the $4$-shape of each polytabloid in the top line is $(2,2)$
and in the bottom line is $(3,1)$. Therefore
\[e(u \cup t) \in -e\!\left(\,\young(135,247,68)\,\right) + e\!\left(\,\young(125,347,68)\,\right) + V^{\unrhd (3,1)},\]
as expected from Proposition \ref{prop: GarnirSplit}.
\end{example}

\begin{proof}[Proof of Theorem \ref{thm: filtration}]
We start by proving
%that each $V^{\unrhd \mu'}$ is a $\Z S_{(m,n)}$-submodule
%of $S^\lambda$ and 
that there is a $\Z S_{(m,n)}$-module isomorphism
\[ \frac{V^{\unrhd \mu'}}{\sum_{\nu' \rhd \mu'} V^{\unrhd \nu'}} \stackrel{\phi}{\cong} S^{\mu} \boxtimes S^{\lambda / \mu}. \]
%\[ \frac{V^{\unrhd \mu'}}{\sum_{\nu' \rhd \mu'} V^{\unrhd \nu'}} \cong S^{\mu} \boxtimes S^{\lambda / \mu}. \]
By Corollary~\ref{cor: SBTmu}, the module on the left-hand side has a $\Z$-basis given by the set of standard $\lambda$-tableaux of $m$-shape $\mu^\prime.$ 
Therefore the linear extension $\phi$ of the map $e(s \cup t) \phi = e(s) \otimes e(t),$
where $s \cup t$ is a standard $\lambda$-tableau of $m$-shape $\mu^\prime,$ is a well-defined $\Z$-linear morphism. 
Since the tensors $e(s) \otimes e(t)$ for $s$ a standard $\mu$-tableau and 
$t$ a standard $\lambda /\mu$-tableau form a basis for $S^\mu \boxtimes S^{\lambda / \mu}$,
$\phi$ is a $\Z$-linear isomorphism.
%$e(s \cup t) \phi = e(s) \otimes e(t)$ defines $\phi$ as a $\Z$-linear isomorphism.
%a $\Z$-linear isomorphism by $e(s \cup t) \mapsto e(s) \otimes e(t)$.
%$\phi : 
%\frac{V^{\unrhd \mu'}}{\sum_{\nu' \rhd \mu'} V^{\unrhd \nu'}} \cong S^{\mu} \boxtimes S^{\lambda / \mu}$
%by $e(s \cup t) \mapsto e(s) \otimes e(t)$. 
%Therefore $\phi$ is a $\Z$-linear isomorphism.

To show that $\phi$ is a $\Z S_{(m,n)}$-module homomorphism,
%compatible with the action of $S_{(m,n)}$, 
it suffices to consider the actions of
$S_{\{1,\ldots, m\}}$ and $S_{\{m+1,\ldots, m+n\}}$ separately. Let $\pi \in S_{\{1,\ldots, m\}}$
and let $s \cup t$ be a standard $\lambda$-tableau. % of $m$-shape $\mu$.
Observe that $\ewidetilde{(s \cup t)\pi}{6.5pt} = \widetilde{s\pi} \cup t$ and
$\epsilon_{(s\cup t)\pi} = \epsilon_{s \pi}$. %Set $u = \widetilde{s \pi}$.
Suppose that $e(\widetilde{s\pi}) = \sum_S \alpha_S e(S)$ where the sum is over all standard $\mu$-tableaux $S$.  
On the one hand
\[ \bigl( e(s) \otimes e(t) \bigr) \pi 
= -\epsilon_{s\pi} \sum_S \alpha_S e(S) \otimes e(t). \]
On the other hand, by Proposition~\ref{prop: GarnirSplit} we have
\[ e(s \cup t) \pi \in -\epsilon_{s\pi}\sum_S \alpha_S e(S \cup t) 
+ \sum_{\nu' \rhd \mu'} V^{\unrhd \nu'}. \]
The argument is entirely analogous for the action of $S_{\{m+1,\ldots, m+n\}}$.

We now write $\ge$ for the lexicographic order of compositions. 
We define $V^{\ge \mu'}$ in a similar way to $V^{\unrhd \mu'},$ replacing the condition $\delta  \unrhd \mu'$ with $\delta \ge \mu'.$ 
Since $\nu' \unrhd \mu'$ implies that $\nu' \ge \mu',$ 
replacing every instance of $\unrhd$ with $\ge$ in Proposition \ref{prop: colStdToStd} and Corollary \ref{cor: SBTmu} implies that $V^{\ge \mu'}$ is also a $\Z S_{(m,n)}$-module. 
Moreover, $V^{\ge \mu'}$ has a $\Z$-basis given by the standard $\lambda$-tableaux of $m$-shape $\nu'$ such that $\nu' \ge \mu',$ and so there is an isomorphism
\[ \frac{V^{\ge \mu'}}{\sum_{\nu' > \mu'} V^{\ge \mu'}} \cong \frac{V^{\unrhd \mu'}}{\sum_{\nu' \rhd \mu'} V^{\unrhd \nu'}} \cong S^{\mu} \boxtimes S^{\lambda / \mu}. \]
Therefore the modules $V^{\ge \mu'},$ where $\mu$ ranges over all subpartitions of $\lambda$ of size $m,$ give the required chain of submodules. 
\end{proof}

\begin{corollary}\label{cor: reduction}
Let $\rho \in S_{m+n}$ be an $n$-cycle and let $\pi$ be a permutation of the remaining
$m$ numbers. Then
\[ \chi^\lambda(\pi \rho) = \sum_\mu \chi^\mu(\pi) \chi^{\lambda/\mu}(\rho) \]
where the sum is over all subpartitions $\mu$ of $\lambda$ of size $m$.
\end{corollary}

\begin{proof}
By taking a suitable conjugate of $\pi\rho$ we may assume that $\pi \in S_{\{1,\ldots, m\}}$
and $\rho \in S_{\{m+1,\ldots, m+n\}}$.
Taking characters in Theorem~\ref{thm: filtration} gives
\begin{equation}
\label{eq: charRes} 
\chi^\lambda \Res_{S_{(m,n)}} = \sum_\mu \chi^\mu \times \chi^{\lambda / \mu}
\end{equation}
where the sum is over all subpartitions $\mu$ of $\lambda$ of size $m$.
Now evaluate both sides at $\pi \rho$.
\end{proof}

%\subsection{The dominance order on skew-tableaux}\label{sec: back_order}

\section{Pieri's rule and Young's rule}\label{sec: Pieri}
A skew partition $\lambda /\mu$ is a \emph{vertical} (resp.~\emph{horizontal}) \emph{strip} 
if $[\lambda/\mu]$ has at most one box in each row (resp.~column). 
Given $n \in \N,$ we write $\sgn_{S_n}$ for the character and the $\C S_n$-module afforded by the sign representation of $S_n$

\begin{theorem}[Pieri's rule]\label{thm: Pieri}
Let $\lambda$ be a partition of $m + n$. If $\mu$ is a subpartition of $\lambda$
of size $m$ then
\[ \langle \chi^\lambda \Res_{S_m \times S_n}, \chi^\mu \times \sgn_{S_n} \rangle = \begin{cases}
1 & \text{if $\lambda / \mu$ is a vertical strip} \\
0 & \text{otherwise.} \end{cases} \]
%\[ (\chi^\lambda \times \sgn_{S_m})\Ind_{S_n \times S_m}^{S_{n+m}} = \sum_\nu \chi^\nu \]
%where is the sum is over all partitions $\nu$ such that $[\nu]$ is obtained from $[\lambda]$ by
%adding $m$ boxes, no two in the same row.
\end{theorem}

\begin{proof}
By Maschke's Theorem (see \cite[(10.8)]{CR}) and~\eqref{eq: charRes}, applied to a suitable conjugate of $S_m \times S_n,$ 
it suffices
to prove that the multiplicity of $\sgn_{S_n}$ as a
direct summand of $S^{\lambda / \mu} \otimes_\Z \C$ is $1$ if $\lambda / \mu$ is a vertical
strip and otherwise $0$. For this we use the corresponding idempotent
$E %_{\sgn_{S_n}} 
= \mfrac{1}{n!} \sum_{\tau \in S_n}
\tau \sgn(\tau) \in \C S_n$. %There are two cases.
%\begin{itemize}
%\item[(a)] 

If $\lambda / \mu$ is not a vertical strip then it contains boxes $(i,j)$, $(i,j+1)$
in the same row. If $t$ is a $\lambda/\mu$-tableau then $\{t\} (1-(x,y)) = 0$ where $x = (i,j)t$ and $y =
(i,j+1)t$. Since $E = \frac{1}{n!}\bigl( 1 - (x,y) \bigr) \sum_\pi \pi \sgn(\pi)$, where the sum is
over a set of right coset representatives for the cosets of $\langle (x,y) \rangle$ in $S_n$,
it follows that $M^{\lambda/\mu} E = 0.$ Hence $S^{\lambda /\mu}E = 0$ as required.
%\item[(b)] 

Suppose that $\lambda / \mu$ is a vertical strip, and let $t$ be
a $\lambda / \mu$-tableau. Let $Y_1, \ldots, Y_c$ be the sets of
entries in each column of $t$. Let $G = S_{Y_1} \times \cdots \times S_{Y_c}$
and let $\pi_1, \ldots, \pi_d$ be a set of right coset representatives
for the cosets of $G$ in $S_n$. Observe that
\[ \bigl\{ (Y_1 \pi_j, \ldots, Y_c \pi_j) : 1 \le j \le d \bigr\} \]
is the complete set of set compositions of $\{1,\ldots, n\}$ into $c$ non-empty parts of 
sizes $|Y_1|, \ldots, |Y_c|$. Let $M = |Y_1|!\ldots |Y_c|!$.
By~\eqref{eq: column},
$e(t) \tau = \sgn(\tau) e(t)$ for each $\tau \in G$. The observation now implies that
\[ e(t) E = \frac{M}{n!} \sum_{i=1}^d   \sgn(\pi_i) e(t \pi_i) \] 
is non-zero and depends on $t$ only up to a sign.
Hence the multiplicity of $\sgn_{S_n}$ in $S^{\lambda / \mu}$ is $1$.
%\end{itemize}
This completes the proof.
\end{proof}

For example, the unique submodule of $S^{(2,1,1) / (1)} \otimes_\Z \C$ affording $\sgn_{S_3}$
is spanned by $e(t) E = \frac{1}{3}e(t) - \frac{1}{3}e\bigl( t (1,2) \bigr) + \frac{1}{3}e\bigl( t (1,3,2) \bigr)$ where
\[ t = \young(:1,2,3)\,,\quad t(1,2) = \young(:2,1,3)\,,\quad t(1,3,2) = \young(:3,1,2). \]
The following lemma is also used in \S 6.

\begin{lemma}\label{lemma: JP}
Let $\lambda$ be a partition of $m + n$ and let $\mu$ be a subpartition of $\lambda$
of size $m.$ 
If $\psi$ is a character of $S_n$ then
\[ \langle \chi^{\lambda / \mu}, \psi \rangle_{S_n} = \big\langle \chi^\lambda, 
 \chi^\mu \times \psi \Ind^{S_{m+n}}_{S_m\times S_n} \bigr\rangle_{\raisebox{-1pt}{$\scriptstyle S_{m+n}$}}. \]
\end{lemma}

\begin{proof}
By Frobenius reciprocity (see \cite[Theorem 38.8]{CR}) and Corollary~\ref{cor: reduction},
\begin{align*}
\langle \chi^{\lambda}, \chi^\mu \times \psi \Ind^{S_{m+n}}_{S_m\times S_n} \rangle &= \langle \chi^{\lambda}\Res^{S_{m+n}}_{S_m\times S_n}, \chi^\mu \times \psi \rangle\\
&= \langle \sum_\nu\chi^\nu \times \chi^{\lambda/\nu}, \chi^\mu \times \psi \rangle 
\end{align*}
where the sum runs over all partitions $\nu$ of $m$ such that $\nu \subset \lambda$.
The only non-zero summand is $\langle \chi^\mu \times \chi^{\lambda/\mu}, \chi^\mu \times \psi\rangle
= \langle \chi^{\lambda/\mu}, \psi\rangle$.
\end{proof}

Using Lemma~\ref{lemma: JP} we immediately obtain the
more usual statement of Pieri's rule that if $\nu$ is a partition
 of $n$ then $(\chi^\nu \times \sgn_{S_\ell})\ind_{S_n \times S_\ell}^{S_{n+\ell}} = \sum_\kappa \chi^\kappa$
 where the sum is over all partitions $\kappa$ of $n+\ell$ such that $\kappa/\nu$ is a vertical strip.
Multiplying by the sign character using the basic result that \smash{$\chi^\nu \times \sgn_{S_n} = \chi^{\nu'}$}
(see for instance \cite[(6.6)]{J}) then gives Young's rule: 
$(\chi^\nu \times 1_{S_\ell})\ind_{S_n \times S_\ell}^{S_{n+\ell}} = \sum_\kappa \chi^\kappa$
 where the sum is over all partitions $\kappa$ of $n+\ell$ such that $\kappa/\nu$ is a horizontal strip.
 
 \begin{remark} A similarly explicit proof of Young's rule can be given, using a similar
 argument to the proof of Theorem~\ref{thm: Pieri}. To reduce to horizontal strips, observe
 that if $t$ is a standard $\lambda/\mu$-tableau with boxes $(i,j)$ and $(i+1,j)$
 then $e(t) \bigl( 1 + (x,y) \bigr) = 0$ where $x = (i,j)t$ and $y = (i+1,j)t$.
 \end{remark}

\section{The dominance lemma for skew tableaux}\label{sec: dominance}
The dominance order for tabloids is defined in \cite[Definition 3.11]{J}, or, in
a way more convenient for us, in \cite[Definition 2.5.4]{S}.
We extend it to compare row standard skew tableaux of shape a fixed skew partition.

\begin{definition}\label{defn: dominanceTableau}
Let $t$ be a row standard $\lambda/\mu$-tableau where $|\lambda/\mu| = n$.
Given $1 \le y \le n,$ we define $\sh_{\le y}(t)$ to be the composition $\beta$ such that
\[ \beta_i = \bigl| \{ x : x \in \text{row $i$ of $t$, $x \le y$} \} \bigr| \]
for $1 \le i \le \ell(\lambda)$. 
If $s$ is another row standard $\lambda/\mu$-tableau,
then we say that $s$ \emph{dominates} $t$, and write $s \unrhd t$, if 
$\sh_{\le y}(s) \unrhd \sh_{\le y}(t)$
for all $y \in \{1,\ldots, n\}$, where on the right-hand side
$\unrhd$ denotes the dominance order of compositions defined in \S\ref{sec: filtration}.
\end{definition}

\begin{example} 
The $\unrhd$ order on the row standard $(3,2)/(1)$-tableaux is shown below, with the largest tableau at the top. 
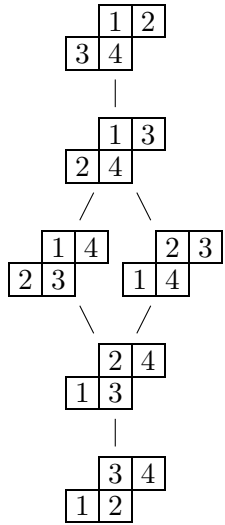
\[\begin{tikzpicture}
\node(1) at (0,0) {\young(:12,34)};
\node(2) at (0,-1.5) {\young(:13,24)};
\node(3) at (.75,-3) {\young(:23,14)};
\node(4) at (-.75,-3) {\young(:14,23)};  
\node(5) at (0,-4.5) {\young(:24,13)};
\node(6) at (0,-6) {\young(:34,12)};
\draw (1) -- (2) -- (3) -- (5) -- (6);
\draw (2) -- (4) -- (5);
\end{tikzpicture}\]
\end{example}

%\begin{remark}\label{rem: order}
%In \cite[\textsection 3]{J}, it is noted that the dominance order can be used to compare tableaux of different shapes and sizes, however this is not the case for the $\unrhd$ order. 
%An order can be defined that allows us to compare row standard skew-tableaux of different shapes and sizes, however this is more than we require.
%\end{remark}

Given a $\lambda/\mu$-tableau $t$, we define its \emph{row straightening} $\bar{t}$ to be the unique
row standard $\lambda/\mu$-tableau whose rows agree setwise with $t$.
We extend the dominance order to $\lambda/\mu$-tabloids by setting $\{s\} \unrhd \{t\}$ if and only if 
$\bar{s} \unrhd \bar{t}.$ 

\begin{lemma}[Dominance Lemma]\label{lemma: dominance}
If $t$ is a column standard $\lambda/\mu$-tableau then $\bar{t}$ is standard and
\[ e(t) = e(\bar{t}) + w, \]
where $w$ is a $\Z$-linear combination of standard polytabloids $e(s)$ such that \hbox{$s \lhd \bar{t}$}.
\end{lemma}
%
%Implicit in the statement of Lemma \ref{lemma: dominance} is that $\bar{t}$ is a standard tableau when $t$ is column standard. 
%In the following lemma, we explicitly prove this claim. 
%
%\begin{lemma}
%Let $t$ be a column standard $\lambda/\mu$-tableau. Then $\overline{t}$ is a standard $\lambda/\mu$-tableau. 
%\end{lemma}
%
%\begin{proof}
%We need to prove that $\overline{t}$ is both row standard and column standard. 
%By construction of $\overline{t},$ it is sufficient to prove that $\overline{t}$ is column standard. 
%
%Suppose, for a contradiction, that there exist boxes $(i,\mu_{i+1}+j)$ and $(i+1,\mu_{i+1}+j)$ in $[\lambda/\mu]$ such that 
%$$(i,\mu_{i+1}+j)\bar{t} > (i+1,\mu_{i+1}+j)\bar{t}.$$
%Define 
%\begin{align*}
%X &= \{(i+1,\mu_{i+1}+1)\bar{t},\ldots,(i+1,\mu_{i+1}+j)\bar{t}\}.\\
%Y &= \{(i,\mu_i+j)\bar{t},\ldots, (i,\lambda_i)\bar{t}\}.
%\end{align*}
%We see that 
%$$\vert X \vert + \vert Y \vert = \lambda_i - \mu_{i+1} + 1 \ge \lambda_i-\mu_i+1,$$
%as $\mu$ is a partition. 
%By the pigeonhole principle, for some $x \in X$ and $y\in Y$, the elements $x$ and $y$ are in the same column of $t$. Therefore
%$$y \ge(i,\mu_{i+1}+j)\bar{t} >(i+1,\mu_{i+1}+j)\bar{t} \ge x.$$
%It follows $t$ is not a column standard tableau, a contradiction. 
%\end{proof}

%\subsubsection*{Preliminaries for the proof of the Dominance Lemma}
We first show that $\bar{t}$ is standard.  Suppose, for a contradiction, that there exist
boxes $(i,j)$ and $(i+1,j) \in [\lambda/\mu]$ such that $(i, j)\bar{t} > (i+1,j)\bar{t}$.
Define
\begin{align*}
R &= \{(i,k)\bar{t} : j \le k \le \lambda_{i} \} \\
S &= \{(i+1,k)\bar{t} : \mu_{i+1} < k \le j \}.
\end{align*}
Since
\[ (i+1,\mu_{i+1}+1)\bar{t} < \ldots < (i+1,j)\bar{t} < (i,j) \bar{t} < \ldots < (i, \lambda_i)\bar{t} \]
we have $x > y$ for each $x \in R$ and $y \in S$. But since $|R| + |S| = \lambda_i - \mu_{i+1} + 1$,
the pigeonhole principle implies that there exist $x \in R$ and $y \in S$ lying in the same column
of the column standard skew tableau $t$, a contradiction.

The next two lemmas generalise Lemmas 3.15~and~8.3 in \cite{J} to skew tableaux. 

\begin{lemma}\label{lemma: order1}
Let $t$ be a $\lambda/\mu$-tableau. 
Let $x$, $y \in \{1,\ldots,n\}$ be such that $x < y$. If $x$ is strictly higher than $y$ in $t$ 
then  $\overline{t(x,y)} \lhd \overline{t}$.
\end{lemma}

\begin{proof}
Let $x$ be in row $k$ of $t$ and let $y$ be in row $\ell$ of $t$. By hypothesis, $k < \ell$.
Let $z \in \{1,\ldots, n\}$.
If $x \le z < y$ then 
\begin{align*}
\sh_{\le z}(\hskip0.5pt\overline{t(x,y)}\hskip0.5pt)_k &= 
\sh_{\le z}(\hskip0.5pt\overline{t}\hskip0.5pt)_k - 1  \\ 
\sh_{\le z}(\hskip0.5pt\overline{t(x,y)}\hskip0.5pt)_\ell &= 
\sh_{\le z}(\hskip0.5pt\overline{t}\hskip0.5pt)_\ell + 1 .
\end{align*}
Whenever $i \not\in \{k, \ell\}$ or $z < x$ or $y \le z$ we have
$\sh_{\le z}(\hskip0.5pt\overline{t(x,y)}\hskip0.5pt)_i = \sh_{\le z}(\hskip0.5pt\overline{t}\hskip0.5pt)_i$.
It easily follows from these equations and the definition of the dominance
order for compositions that
 $\overline{t (x,y)} \lhd \overline{t}$.
\end{proof}

\begin{lemma}\label{lemma: order2}
Let $t$ be a column standard $\lambda/\mu$-tableau. Then 
$e(t) = \{t\} + w$,
where $w$ is a $\Z$-linear combination of $\lambda/\mu$-tabloids $\{s\}$ such that $\{ s\} \lhd \{t\}.$
\end{lemma}
\begin{proof}
The proof of Lemma~8.3 in \cite{J} still holds, replacing Lemma 3.15 in \cite{J} with our 
Lemma~\ref{lemma: order1}.
\end{proof}

%We prove by induction on the dominance order on 
%$\lambda/\mu$-tabloids that if the coefficient of $\lambda/\mu$-tabloid $\{t'\}$ in $e(t)$ is non-zero, then $\{t'\} \unlhd \{t\}.$
%The base case is when $\{t'\} = \{t\},$ where the result clearly holds. 
%
%Suppose now that $\{t'\}$ is a summand of $e(t)$ such that $\{t'\} \neq \{t\},$ and assume that the result holds for all summands $\{s\}$ of $e(t)$ such that $\{s\} \rhd \{t'\}.$ 
%By definition of the polytabloid $e(t)$, there exists a non-identity element $\pi \in C(t)$ such that $\{t'\} = \{t \pi\}.$
%As $t$ is a column standard tableau, there exist natural numbers $w$ and $x$ such that $w < x$ and $w$ is lower than $x$ in $t'$. 
%By Lemma \ref{lemma: order1}, we have that $\{t'\} \lhd \{t(w\ x)\}.$
%As $\{t(w\ x)\}$ is a summand of $e(t),$ it follows from the inductive hypothesis that $\{t(w\ x)\} \unlhd \{t\}.$ 
%Therefore $\{t'\} \lhd \{t\},$ and so the lemma is proved. 
%\end{proof}

\begin{proof}[Proof of Lemma \ref{lemma: dominance}]
Let 
$e(t) = \sum_{s} \alpha_s e(s)$ where the sum is over all standard $\lambda/\mu$-tableaux and
$\alpha_s \in \Z$ for each $s$.
Let $u$ be a standard tableau maximal in the dominance order such that $\alpha_u \neq 0.$ 
Applying Lemma \ref{lemma: order2} to $e(u)$ gives
\[ e(u) = \{u\} + w^{\lhd \{u\}}, \]
where $w^{\lhd \{u\}}$ is a $\Z$-linear combination of $\lambda/\mu$-tabloids each dominated by $\{u\}$. 
By Lemma \ref{lemma: order2} and the maximality of $u$, there is no other standard $\lambda/\mu$-tableau
$s$ with $\alpha_s \not= 0$ such that $e(s)$ has $\{u\}$ as a summand.
Therefore the coefficient of $\{u\}$ in $e(t)$ is $\alpha_u$. 
Applying Lemma \ref{lemma: order2}, now to $e(t),$ gives 
$$e(t) = \{t\} + w^{\lhd \{t\}},$$
where $w^{\lhd \{t\}}$ is a $\Z$-linear combination of $\lambda/\mu$-tabloids 
each dominated by~$\{t\}$.
In particular $\{t\} \unrhd \{u\},$ and so we have that $\bar t = u$ by the maximality of $u.$ 
Hence
\[e(t) = \alpha_{\bar t}e(\bar{t}) + w,\]
where $w$ is a $\Z$-linear combination of standard polytabloids 
$e(v)$ for standard tableaux $v$ such that $v \lhd \bar t$. 
It follows that $\{t\}$ cannot be a summand of $w$ in the equation immediately above. 
Since the coefficient of $\{t\}$ in $e(t)$ is 1, we have $\alpha_{\bar t} = 1.$ 
\end{proof}

We isolate the following corollary of Lemma \ref{lemma: dominance}.
\begin{corollary}\label{cor: dominance}
Let $s$ be a standard $\lambda/\mu$-tableau, and let $u$ be a column standard $\lambda/\mu$-tableau.
Suppose that there exists $x \in \{1,2,\ldots,n\}$ such that the boxes containing 
$1,2,\ldots,x-1$ are the same in $s$ and~$u$, and $x$ is lower in $u$ than $s$.
If
\[ e(u) = \sum \alpha_v e(v), \]
where the sum is over all standard $\lambda$-tableaux $v$, then $\alpha_s = 0$.
\end{corollary}

\begin{proof}
By assumption, $\sh_{\le z}(s) = \sh_{\le z}(\bar{u})$ if $1 \le z < x.$
As $x$ is in a lower row in $u$ than in $s$, we have  $\sh_{\le x}(\bar{u}) \ntriangleright \sh_{\le x}(s)$.
Now apply Lemma~\ref{lemma: dominance}.
\end{proof}

\section{The Murnaghan--Nakayama rule for border strips}
\label{sec: snakeCase}

In this section we give a bijective proof that $\chi^{\lambda/\mu}(\rho) = (-1)^{\height(\lambda / \mu)}$
when $\lambda / \mu$ is a border strip of size~$n$ and $\rho$ is the $n$-cycle
$(1,2,\ldots, n)$. This deals with one of the two cases in~\eqref{eq: skewCase}.
Our proof shows that the matrix representing $\rho$
in the standard basis of $S^{\lambda/\mu}$ has a unique non-zero entry on its diagonal. The relevant
standard tableau is defined as follows.

\begin{definition}\label{defn: columnar}
Let $\lambda / \mu$ be a border strip of size $n$. Say that a box $(i,j) \in [\lambda/ \mu]$
is \emph{columnar} if $(i+1,j) \in [\lambda/ \mu]$.
% or if $(i,j)$ is at the bottom-left of
%$[\lambda/\mu]$ in a row of length $1$. 
We define the standard $\lambda/\mu$-tableau $t_{\lambda / \mu}$ as follows:
\begin{itemize}
\item[(i)] assign the numbers $\{1,\ldots, z\}$ in ascending order
to the $z$ columnar boxes of $\lambda / \mu$, starting with~$1$ in row~$1$ and finishing
with $z$ in the row above the bottom row;
\item[(ii)] then assign 
the numbers $\{z+1,\ldots, n\}$ in ascending order to the $n-z$ non-columnar boxes, starting with $z+1$ in 
column $1$ and finishing with~$n$ in the rightmost column.
\end{itemize}
\end{definition}

For example, $t_{(5,3,3)/(2,2)}$, $t_{(5,3,2)/(2,1)}$ and $t_{(5,1,1)/\varnothing}$ are respectively 
\[ 
\young(::167,::2,345)\, , \quad
\young(::167,:25,34)\, \text{ and }\;
\young(14567,2,3)
\]
where $1$ and $2$ are the entries in columnar boxes in each case.
We remark that there are no columnar
boxes if and only if $\lambda/\mu$ is a horizontal strip, as defined in \S\ref{sec: Pieri}.

As useful pieces of notation, we define
$x^-$ and $x^+$ for $x \in \{1,\ldots, n\}$ by $x^- = x-1$ and
\[ x^+ = \begin{cases} x+1 & \text{if $1 \le x < n$} \\
1 & \text{if $x=n$.} \end{cases} \]
Thus $x\rho = x^+$ for all $x \in \{1,\ldots, n\}$ and $1^- = 0$. 
Given a $\lambda/\mu$-tableau $t$,
we define $t^+$ by $(i,j)t^+ = \bigl((i,j)t)^+$.
By~\eqref{eq:cyclic}, $e(t\rho) = e(t^+)$.

We say that a 
standard $\lambda/\mu$-tableau~$t$ such that $e(t)$ has a non-zero coefficient in the unique
expression of $e(t^+)$ as a $\Z$-linear combination of standard polytabloids
is
\emph{trace-contributing}. 
Since $\chi^{\lambda/\mu}(\rho)$ is the trace of the matrix
representing $\rho$ in the standard basis,
 it suffices to prove the following proposition.

\begin{proposition}\label{prop: unique}
Let $\lambda/\mu$ be a border strip. The unique trace-contributing $\lambda/\mu$-tableau 
is $t_{\lambda/\mu}$. The coefficient
of $e(t_{\lambda/\mu})$ in $e(t_{\lambda/\mu}^+)$ is $(-1)^{\height(\lambda/\mu)}$.
\end{proposition}

The proof of Proposition~\ref{prop: unique} is by induction on the number of top corner
boxes of $\lambda/\mu$, as defined in Definition~\ref{defn: corners} below. The necessary
preliminaries are collected below. We then prove the base case, when $\lambda/\mu = (n-\ell,1^\ell)$
for some $\ell \in \N_0$; this gives a good flavour of the general argument.
In the remainder of this section we give the inductive step.

We assume, without loss of generality, that $\mu_1 < \lambda_1$ and $\mu_{\ell(\lambda)} = 0$,
so the non-empty rows of $\lambda/\mu$ are $1, \ldots, \ell(\lambda)$ and column $1$ of $\lambda/\mu$
is non-empty.
We can do this since the character indexed by a skew diagram is equal to the character indexed by the same skew diagram with its empty rows and columns removed. 

\subsection{Preliminaries for the proof of Proposition~\ref{prop: unique}}
For $Z\hspace{-1.5pt} \subseteq\hspace{-1.5pt} \{1,\ldots, n\}$ 
and $t$ a row standard $\lambda/\mu$-tableau
we define $\sh_Z(t)$ to be the composition $\beta$ such that
\[ \beta_i = \bigl| \{ x : x \in \text{row $i$ of $t$, $x \in Z$} \} \bigr| \]
for $1 \le i \le \ell(\lambda)$. Set $\sh_{< y}(t) = \sh_{\{1,\ldots,y^-\}}(t)$.
We also use $\sh_{\le y}(t)$, as already defined in Definition~\ref{defn: dominanceTableau}.

\begin{definition}\label{defn: corners}
Let $\lambda/\mu$ be a border strip.
We say that column $j$ of $\lambda/\mu$ is \emph{singleton} if it contains a unique box.
We define a \emph{top corner box} 
to be a box $(i,j) \in [\lambda/\mu]$ such that $(i,j-1), (i-1,j) \not\in [\lambda/\mu]$
and a \emph{bottom corner box} to be a box $(i,j) \in [\lambda/\mu]$ such that
$(i+1,j), (i,j+1) \not\in [\lambda/\mu]$.
\end{definition}

\begin{lemma}\label{lemma: horizontal}
Let $\lambda/\mu$ be a border strip and let $t$ be a $\lambda/\mu$-tableau. 
If columns~$j$ and $j+1$ of $\lambda/\mu$ are singleton, with their unique box in row~$i$, 
then $e(t) = e(t) (x,y)$ where $x = (i,j)t$ and $y = (i,j+1)t$.
\end{lemma}

\begin{proof}
This follows immediately from the Garnir relation~\eqref{eq: Garnir}, taking $X = \{x\}$ and $Y = \{y\}$.
\end{proof}

In fact, all the Garnir relations that we use can  be reduced to single transpositions.
Let $x$ and~$y$ be entries in adjacent columns of a column standard tableau,
with $x$ left of $y$ and $x > y $. We say that $(x,y)$ is a \emph{Garnir swap}
if at least one of these column is not singleton, and otherwise that $(x,y)$ is a \emph{horizontal~swap}.

\begin{lemma}\label{lemma: straightening}
Let $t$ be a trace-contributing border strip tableau. 
Then $t$ can be obtained from $\widetilde{t^+}$ by iterated horizontal swaps, Garnir swaps
and column straightenings.
If in such a sequence $1$ moves, then $1$ moves either left or down.
\end{lemma}

\begin{proof}
The first claim is immediate from~Theorem~\ref{thm: SBT}(i).
The second follows from Corollary~\ref{cor: dominance} taking $x=1$.
%, since if $s$ and $t$ are standard tableaux
%with $s \unrhd t$ then $1$ appears at least as high in $s$ as $t$.
\end{proof}

Given $X \subseteq \{1,2,\ldots,n\},$ we define $X^+ = \{x^+ : x \in X\}.$ 
We also define $\min X$ to be the minimum of $X,$ and $\max X$ to be the maximum of $X.$ 
The following combinatorial result on the map $x \mapsto x^+$ is used several times to 
restrict the possible entries of trace-contributing tableaux.

\begin{lemma}\label{lemma: shiftset}
Let $X$ be a set of natural numbers such that $1,n \not\in X.$ 
Also suppose that $b, c$ are not contained in $X.$ 
We have $\{b^+\} \cup X^+ = X \cup \{c\}$ if and only if
$b^+ = \min X$, $c = \max X^+$ and~$X = \{b^+, \ldots, c^-\}$.
\end{lemma}

\begin{proof}
Since $\min X \not\in X^+$ we have $\min X = b^+$. Similarly,
since $\max X^+ \not\in X$ we have $\max X^+ = c$. 
Suppose for a contradiction
that $X$ is a proper subset of $\{b^+,\ldots,c^-\}$. Setting 
\[ d = \min ( \{b^+,\ldots,c^-\} \backslash X) \]
we see that since $b^+ = \min X \in X$, we have $d > b^+$. The minimality
of $d$ implies that $d^- \in X$ and so $d \in X^+$; since 
$d < c$ and $\{b^+\} \cup X^+ = X \cup \{c\}$, we have $d \in X$, a contradiction.
The converse is obvious.
\end{proof}

Finally, as a notational convention, when we specify a set, we always list the elements
in increasing order. In diagrams the symbol $\star$ marks an entry we have no need
to specify more explicitly.

\subsection{Base case: one top corner box}

In this case $\mu = \varnothing$
and $\lambda = (n-\ell,1^\ell)$ for some $\ell \in \N_0$. If $\ell=0$ then there is a unique standard
$(n)$-tableau and the result is clear. Suppose that $\ell > 0$ and 
let $t$ be a standard $(n-\ell,1^\ell)$-tableau with entries 
$\{1,y_1,\ldots, y_{\ell-1},c\}$ in column $1$.
(By our notational convention, $1 < y_1 < \ldots < y_{\ell-1} < c$.) If $c = n$ then
$\widetilde{t^+}$ is standard
with first column entries $\{ 1, 1^+, y_1^+, \ldots, y_{r-1}^+\}$.
Hence, assuming that $t$ is trace-contributing, we have $c < n$.
After a sequence of horizontal swaps applied to $\ewidetilde{t^+}{8pt}$ we obtain the tableau shown below.

\smallskip
\begin{center}
\begin{tikzpicture}[x=0.85cm, y=0.75cm]
\tbox{1}{1}\tlabel{1}{1}{$1^+$}
\tbox{1}{2}\tlabel{1}{2}{$1$}
\tbox{1}{3}\tlabel{1}{3}{$\star$}
\tput{1.5}{4.5}{$\ldots$}
\tbox{1}{5}\tlabel{1}{5}{$\star$}

\tbox{2}{1}\tlabel{2}{1}{$y_1^+$}
\tput{3.375}{1.5}{$\vdots$}
\tbox{4}{1}\tlabel{4}{1}{$y_{\ell-1}^+$}
\tbox{5}{1}\tlabel{5}{1}{$c^+$}
\end{tikzpicture}
\end{center}

A Garnir swap of $1$ with $1^+$ or any $y_i^+$ gives, after column straightening and a sequence of horizontal swaps, a standard
tableau having $c^+$ in its bottom left position. We may therefore assume, by Lemma~\ref{lemma: straightening}, that $1$ is swapped with $c^+$.
After column straightening, which introduces the sign $(-1)^\ell$,
a sequence of horizontal swaps gives the standard tableau having $\{1,1^+, y_1^+, \ldots, y_{\ell-1}^+\}$ in its first column. Thus if $t$ is trace-contributing
then $\{1^+, y_1^+, \ldots, y_{\ell-1}^+ \} = \{y_1,\ldots, y_{\ell-1}, c\}$.
By Lemma~\ref{lemma: shiftset}, $\{y_1,\ldots, y_{\ell-1}, c\} = \{2, \ldots, \ell+1\}$. Therefore
$t = t_{(n-\ell,1^\ell)}$ and the coefficient of $e(t_{(n-\ell,1^\ell)})$ in $e(t_{(n-\ell,1^\ell)}^+)$ is $(-1)^\ell$, as required.

\subsection{Inductive step}
Let $\delta(i) \in \N_0^{\ell(\lambda)}$ denote the composition
defined by $\delta(i)_i = 1$ and $\delta(i)_{k} = 0$ if $k \not= i$.

\begin{proposition}\label{prop:ieqip}
Let $\lambda/\mu$ be a border strip, and let $t$ be a standard $\lambda/\mu$-tableau. Let $c \in \N$ and suppose 
that \emph{either} %\emph{(i)} 
$c=1$ \emph{or} $c > 1$ and %\emph{(ii)} 
the entries $1,\ldots, c^-$ and $n$ lie in the
same column of $t$.
%\begin{itemize}
%\item[(a)] $c=1$; \emph{or}
%\item[(b)] the entries $1,\ldots, c^-$ and $n$ lie in the 
%same column of $t$.
%\end{itemize}
Let $(i,j)$ be the box of $t$ containing $c$, and let $(i',j')$ be the
box of $\widetilde{t^+}$ containing $c$. 
If $t$ is a trace-contributing tableau, then $i = i'$.
\end{proposition}

\begin{proof}
By hypothesis, 
the highest $c^-$ entries in column $j'$ of $t$ and $\widetilde{t^+}$ are 
$1, \ldots, c^-$. Let $s = \widetilde{t^+}$.
Setting $\beta = \sh_{< c}(t) = \sh_{< c}(\bar{s})$ we
have $\sh_{\le c}(t) = \beta + \delta(i)$ and $\sh_{\le c}(\bar{s}) = \beta + \delta(i')$.
By Lemma~\ref{lemma: dominance}, the hypothesis that $t$ is trace-contributing implies that
$\sh_{\le c}(\bar{s}) \unrhd \sh_{\le c}(t)$. Therefore $i \ge i'$.

If $j = j'$ then \emph{either} $c=1$ and $1$ is at the top of the column of $t$
which has $n$ at its bottom, \emph{or} $c > 1$ and $c$ is immediately below $c^-$ 
in both $s$ and~$t$. In either case $i=i'$.

We may therefore suppose, for a contradiction, that $i > i'$ and $j < j'$.
By  hypothesis the box $(i,j)$ of $t$ containing $c$ is the top corner box in row~$i$.
Let $(i,\ell)$ be the bottom corner box in row $i$; note that $\ell \le j'$,
as shown in the diagram below.
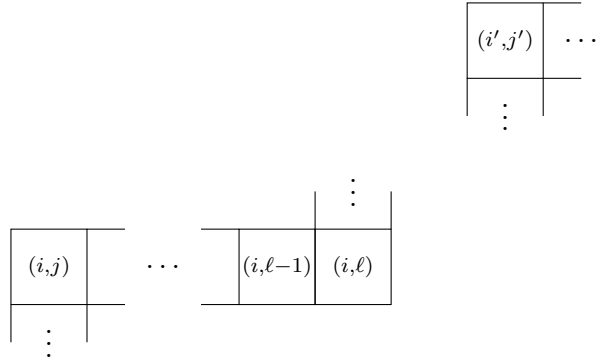
\begin{center}
\begin{tikzpicture}[x=1cm, y=1cm]

\tbox{0}{7}\tlabel{0}{7}{$\scriptstyle (i',j')$}
\draw (8,0)--(8.5,0);
\draw (8,-1)--(8.5,-1);
\draw (7,-1)--(7,-1.5);
\draw (8,-1)--(8,-1.5);

\tput{1.4}{7.5}{$\vdots$}
\tput{0.5}{8.5}{$\ldots$}
\tput{2.4}{5.5}{$\vdots$}
\tput{4.4}{1.5}{$\vdots$}

\tbox{3}{5}\tlabel{3}{5}{$\scriptstyle (i,\ell)$}
\tbox{3}{4}\tlabel{3}{4}{$\scriptstyle (i,\ell-1)$}
\tput{3.5}{3}{$\ldots$}
\tbox{3}{1}\tlabel{3}{1}{$\scriptstyle (i,j)$}
\draw (4,-3)--(3.5,-3);
\draw (4,-4)--(3.5,-4);

\draw (5,-3)--(5,-2.5);
\draw (6,-3)--(6,-2.5);

\draw (2,-3)--(2.5,-3);
\draw (2,-4)--(2.5,-4);

\draw (1,-4)--(1,-4.5);
\draw (2,-4)--(2,-4.5);

\end{tikzpicture}
\end{center}

By the hypothesis that $t$ is trace-contributing and Lemma~\ref{lemma: straightening} 
there is a sequence of horizontal swaps, Garnir swaps, and column straightenings from $\widetilde{t^+}$
to $t$. Suppose that in such a sequence an entry $b < c$ is moved. If~$b$ is the 
first such entry moved in this sequence, and $u$ is the tableau obtained after column straightening, 
then, by Corollary~\ref{cor: dominance}
applied with $x = b$, the coefficient of $e(t)$ in $e(u)$ is zero.  
%$\sh_{\le b}(t) \rhd \sh_{\le b}(v)$. Hence $\sh(v) \notunrhd \sh(t)$ and,
%by Lemma~\ref{lemma: dominance}, $e(t)$ does not appear in $e(v)$, a contradiction.
Therefore the entries $\{1,\ldots,c^-\}$ are fixed and $c$ is the smallest number moved.
Take such a sequence and stop
it immediately after 
the first swap in which $c$ enters
row $i$. Let $v$ be the column standard
tableau so obtained, and let $u$ be its immediate predecessor. 
%It suffices to show that $e(t)$ does not appear in $e(v)$.

When $c$ enters row $i$ of $v$, it is 
swapped with the entry, $d^+$ say, in box $(i,\ell-1)$ of $u$. 
Observe that the entries in boxes 
strictly to the left of column $\ell$ are the same in $\widetilde{t^+}$ and $u,$
since no swap in the sequence from $\widetilde{t^+}$ to $u$
involves an entry in these columns.
Let $a^+$ be the entry in box $(i,\ell)$ of $u$. Thus the column standard tableau
$u$ is as shown overleaf and $v = \ewidetilde{u(c,d^+)}{7pt}$.

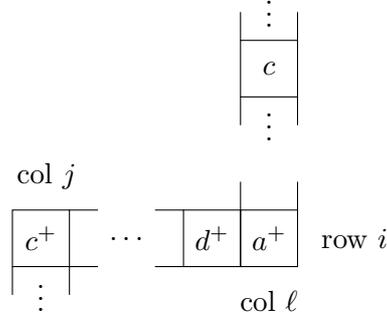
\begin{center}
\begin{tikzpicture}[x=0.75cm, y=0.75cm]
\tput{-0.625}{5.5}{$\vdots$}
\tbox{0}{5}\tlabel{0}{5}{$c$}
\tput{1.375}{5.5}{$\vdots$}
\tbox{3}{5}\tlabel{3}{5}{$a^+$}
\tbox{3}{4}\tlabel{3}{4}{$d^+$}
\tput{3.5}{3}{$\ldots$}
\tbox{3}{1}\tlabel{3}{1}{$c^+$}
\draw (4,-3)--(3.5,-3);
\draw (4,-4)--(3.5,-4);

\draw (5,-3)--(5,-2.5);
\draw (6,-3)--(6,-2.5);

\draw (5,-1)--(5,-1.5);
\draw (6,-1)--(6,-1.5);

\draw (5,0)--(5,0.5);
\draw (6,0)--(6,0.5);

\draw (2,-3)--(2.5,-3);
\draw (2,-4)--(2.5,-4);

\draw (1,-4)--(1,-4.5);
\draw (2,-4)--(2,-4.5);

\tput{3.5}{7}{row $i$}
\tput{4.4}{1.5}{$\vdots$}

\tput{4.6}{5.5}{col $\ell$}
\tput{2.4}{1.6}{col $j$}

\end{tikzpicture}
\end{center}

\noindent Note that $d^+ > a^+$ since otherwise $u$
is standard with respect to all boxes weakly to the left of column $\ell$,
and so $d^+$ cannot be moved in a Garnir swap.
%Moreover, $d^+ > a^+ > c$. 

To complete the proof we require the following critical quantity.
\emph{Let $r$ be maximal such that entries $c, \ldots, r$ are strictly to the
left of column $\ell$ in the original tableau $t$.}
If $r = d$ then, since $d > a$, $a$ is strictly to the left of column $\ell$ in $t$;
this is impossible since $a^+$ appears in column $\ell$ in $u$. Therefore $r < d$.
Since $d$ is in position $(i,\ell-1)$ of $t$ and $r \ge c$, it follows that 
$c \not = d$. Moreover, the entries $c^+, \ldots, r^+$ are in the same
boxes in $t^+$ and $v.$

%(Equivalently, row $i$ has length at least three.)

%We use this to prove the following claim.
\emph{Claim.} We have $v \notunrhd t$. 
\noindent\emph{Proof of claim.} 
Let $\sh_{\{c^+, \ldots, r^+\}}(u) = \delta$. 
By hypothesis and our stopping condition on swaps,
if $q \le r$ then the box of $q^+$ in $u$ is the box of $q$ in $t$.
Hence  $\sh_{\{c,\ldots,r\}}(t) = \delta$.
Since $d > r$ and $d$ is in position $(i,\ell-1)$ of $t$, we see that
$r^+$ is not in row~$i$ of $t$. By maximality of $r$, the row of $t$
containing $r^+$
is row $h$ for some $h < i$. Clearly the row of $c$ in $v$ is $i$.
Therefore $\sh_{\{c,\ldots,r^+\}}(\bar{v}) = \delta + \delta(i)$ and
$\sh_{\{c,\ldots,r^+\}}(t) = \delta + \delta(h)$. Since $1,\ldots, c^-$ are in
the same positions in both $v$ and $t$, it follows that
\[ \sh_{\le r^+}(t) \rhd \sh_{\le r^+}(\bar{v}) \]
which implies the claim. %$\|$

It now follows from Lemma~\ref{lemma: dominance}, as before, that $e(t)$
does not appear in $e(v)$, a final contradiction. This completes the proof.
\end{proof}

%\begin{corollary}\label{cor:ieqip}
%Let $t$ be a trace-contributing tableau. 
%Let $i$ be the row of $1$ in $t$ and let $i'$ be
%the row of $1$ in $\widetilde{t^+}$. Then $i = i'$
%\end{corollary}
%
%\begin{proof}
%This is immediate from Proposition~\ref{prop:ieqip} taking $c=1$.
%\end{proof}

\begin{corollary}\label{cor:toprow}
If $t$ is a trace-contributing tableau then either $1$ and $n$
are in the same column of $t$, or $1$ and $n$ are in the top row of $t$.
\end{corollary}

\begin{proof}
Let $1$ and $n$ be in positions $(i,j)$ of $t$ and $(i',j')$ of $t$,
respectively. If column $j'$ is singleton then $n$ is the top right
entry of $t$ and, taking $c=1$ in Proposition~\ref{prop:ieqip}, we get $i=i'$;
thus $1$ and $n$ are in the top row of~$t$. Otherwise, when we column
straighten $t^+$ to obtain $\widetilde{t^+}$, the entry $1$
in position $(i',j')$ moves up to position $(i'',j')$ where $i'' < i'$.
Again taking $c=1$ in Proposition~\ref{prop:ieqip}, we get $i=i''$. Since $(i'',j')$ is
the top corner box in its row, and so is $(i,j)$, we see that $j=j'$. Hence
$1$ and $n$ are in the same column of $t$.
\end{proof}

\begin{proof}[Proof of Proposition \ref{prop: unique}]
We now complete the inductive step of the proof. 

Suppose that $\lambda/\mu$ has more than one top corner box, 
and that $t$ is a trace-contributing $\lambda/\mu$-tableau.  
Let 1 be in position $(i,j)$ of $t$ and in position $(i',j')$ of $\widetilde{t^+}$. 
By Proposition \ref{prop:ieqip}, we have  $i = i'.$
%We  distinguish two cases. 

\smallskip
{\it Case} (1). Suppose that $1$ and $n$ lie in the same row of $t$. By Corollary~\ref{cor:toprow},
this is the top row. Let the entries in the top row be $\{1, x_1, \ldots, x_{k-1}, n \}$,
and let the entries in the column of $1$ be $\{1, y_1, \ldots, y_{\ell-1}, c\}$.

Straightening the top row of $t^+$ by a sequence of $k-1$
horizontal swaps moves~$1^+$ and $1$ into
adjacent positions, giving the tableau $u$ shown below.

\begin{center}
\begin{tikzpicture}[x=0.85cm, y=0.75cm]
\tbox{1}{1}\tlabel{1}{1}{$1^+$}
\tbox{1}{2}\tlabel{1}{2}{$1$}
\tbox{1}{3}\tlabel{1}{3}{$x_1^+$}
\tput{1.5}{4.5}{$\ldots$}
\tbox{1}{5}\tlabel{1}{5}{$x_{k-1}^+$}

\tbox{2}{1}\tlabel{2}{1}{$y_1^+$}
\tput{3.375}{1.5}{$\vdots$}
\tbox{4}{1}\tlabel{4}{1}{$y_{\ell-1}^+$}
\tbox{5}{1}\tlabel{5}{1}{$c^+$}

\draw (1,-5)--(0.5,-5);
\draw (1,-6)--(0.5,-6);

\tput{5.5}{0.5}{$\ldots$}

\end{tikzpicture}
\end{center}

As in the base case, the only Garnir swap that can lead to $t$ is $(1,c^+)$,
which introduces the sign $(-1)^\ell$. Let $v = \ewidetilde{u(1,c^+)}{8pt}$, as shown below.

\begin{center}
\begin{tikzpicture}[x=0.85cm, y=0.75cm]
\tbox{1}{1}\tlabel{1}{1}{1}
\tbox{1}{2}\tlabel{1}{2}{$c^+$}
\tbox{1}{3}\tlabel{1}{3}{$x_1^+$}
\tput{1.5}{4.5}{$\ldots$}
\tbox{1}{5}\tlabel{1}{5}{$x_{k-1}^+$}

\tbox{2}{1}\tlabel{2}{1}{$1^+$}
\tbox{3}{1}\tlabel{3}{1}{$y_1^+$}
\tput{4.375}{1.5}{$\vdots$}
\tbox{5}{1}\tlabel{5}{1}{$y_{\ell-1}^+$}

\draw (1,-5)--(0.5,-5);
\draw (1,-6)--(0.5,-6);

\tput{5.5}{0.5}{$\ldots$}

\end{tikzpicture}
\end{center}

By Lemma~\ref{lemma: straightening} and Corollary~\ref{cor: dominance}, 
$v$ can be straightened by a sequence of horizontal swaps, Garnir swaps and column straightenings which
either fix $1$, and so leave invariant the content of its top row,
or move $1$ into a lower row, giving a tableau, $w$ say, such that,
$e(t)$ does not appear in $e(w)$. Since $e(t)$ has a non-zero coefficient in $e(v)$, we have
\[ \{c^+, x_1^+, \ldots, x_{k-1}^+ \} = \{x_1, \ldots, x_{k-1}, n \}. \]
Lemma~\ref{lemma: shiftset} implies that $c^+ = x_1 = n-k+1$, $x_{k-1}^+ = n$ and
$\{x_1,\ldots, x_{k-1}\} = \{n-k+1, \ldots, n-1\}$. 
Thus $t$ and $v$ have top row entries $\{1, n-k+1, \ldots, n\}$.

Let $T$ and $V$ be the tableaux obtained from $t$ and $v$ by deleting
all but the top corner box in their top rows. This removes entries $\{n-k+1,\ldots, n\}$.
Let $\lambda^\star/\mu$ be the common shape of $T$ and $V$.
Observe that $T$ has greatest
entry $n-k = c$ in the bottom corner box of its rightmost column
and that $V$ is the column straightening of $T^\dagger$, where 
$\dagger$ is defined as $+$ on tableaux, but replacing $n$ with $n-k$.
By induction, $T = t_{\lambda^\star/\mu}$, and since $t$
has $n-k+1, \ldots, n$ in its top row, we have $t = t_{\lambda/\mu}$.
Moreover, the coefficient of $e(T)$ in $e(T^\dagger)$ is $(-1)^{\height(\lambda^\star/\mu)}$,
Since 
%MW surely the height does not change? JK quite right, it doesn't.
$\height(\lambda^\star/\mu) = \height(\lambda/\mu)$, the
%$\height(\lambda^\star/\mu) + \ell = \height(\lambda/\mu)$, the
coefficient of $e(t)$ in $e(t^+)$ is $(-1)^{\height(\lambda/\mu)}$, as required.

\smallskip
{\it Case} (2). If Case (1) does not apply then, since $i=i'$, 
$1$ and $n$ are in the same column of $s$ and so $j = j'$.
Take $c$ maximal such that $1,2,\ldots, c^-$
are in column $j$ of $t.$ 
Suppose that in column $j$ of $t,$ the entry immediately below $c^-$ 
equals $d$ for some $d < n.$ 
By Proposition~\ref{prop:ieqip}, the row of $c$ in $t$
is the same as the row of $c$ in $\widetilde{t^+}$. 
It follows that $c = d,$ which contradicts the maximality of $c$ unless 
%By the maximality of $c$ it follows that 
column $j$ of $t$ has entries $1,2, \ldots, c^-, n$, as shown below.

\begin{center}
\begin{tikzpicture}[x=0.75cm, y=0.75cm]
\tbox{1}{1}\tlabel{1}{1}{1}

\tbox{2}{1}\tlabel{2}{1}{$2$}
\tbox{4}{1}\tlabel{4}{1}{$c^-$}
\tput{3.375}{1.5}{$\vdots$}
\tbox{5}{1}\tlabel{5}{1}{$n$}

\tbox{5}{-3}\tlabel{5}{-3}{$c$}
\draw (-2,-5)--(-1.5,-5);
\draw (-2,-6)--(-1.5,-6);
\tput{5.5}{-1.5}{$\ldots$}

\draw (2,-1)--(2.5,-1);
\draw (2,-2)--(2.5,-2);
\tput{1.5}{2.5}{$\ldots$}

\draw (-3,-6)--(-3,-6.5);
\draw (-2,-6)--(-2,-6.5);
\tput{6.375}{-2.5}{$\vdots$}

\draw (1,-5)--(0.5,-5);
\draw (1,-6)--(0.5,-6);

\tput{5.5}{0.5}{$\ldots$}

\tput{5.5}{4}{row $i = i'$}

\tput{6.6}{1.5}{col $j$}

\end{tikzpicture}
\end{center}

By Lemma~\ref{lemma: straightening} there is a 
sequence of horizontal swaps, Garnir swaps and column straightenings from $\widetilde{t^+}$ to $t$. 
As seen in the proof
of Proposition~\ref{prop:ieqip}, it follows easily from Lemma~\ref{lemma: dominance}
that $1, \ldots, c^-$ do not move.
Let $X$ be the set of entries of $t$ lying strictly to the right of column $j$.
These entries become $X^+$ in $\widetilde{t^+}$, which is standard with respect to these columns.
No permutation in our chosen sequence can involve a entry in one of these columns. Hence $X^+ = X$,
and so $X = \varnothing$.

We have shown that $j$ is the rightmost column of $t$, and that $t$ agrees with
$t_{\lambda/\mu}$ in this column.
Let $T$ be 
the tableau obtained from $t$ by deleting all but the bottom corner box in column $j$
and subtracting $c^-$ from each remaining entry. Thus the top row of $T$ has entries
$1, \ldots, n-c^-$ and $n-c^-$ is its greatest entry. Let $T$ have shape
 $\lambda^\star/\mu^\star$.
By induction, $T = t_{\lambda^\star/\mu^\star},$ and hence $t = t_{\lambda/\mu}$. 
Let $T^\dagger$ be defined as $T^+$, but replacing $n$ with $n-c^-$.
By induction, the coefficient of $e(T)$ in $e(T^\dagger)$,
is $(-1)^{\height(\lambda^\star/\mu^\star)}$.
Since $\height(\lambda^\star/\mu^\star) + c^- = \height(\lambda/\mu)$, and
the sign introduced by column straightening $t^+$ is $(-1)^{c^-}$,
the coefficient of $e(t)$ in $e(t^+)$ is $(-1)^{\height(\lambda/\mu)}$, as required.
\end{proof}
\section{Proof of Theorem~\ref{thm: MN rule}}
\label{sec: endgame}

Let $\lambda/\mu$ be a skew partition of size $n$ and let $\rho \in S_n$ be an $n$-cycle.
In order to complete the proof of Theorem~\ref{thm: MN rule}, we must show that
$\chi^{\lambda / \mu}(\rho) = 0$ if $\lambda / \mu$ is not a border strip. 
We require the following two lemmas.
\vspace*{-5pt}
\begin{lemma}\label{lemma: YP}
Let $0\le \ell \le n$. If
\[ \langle \chi^{\lambda}, \chi^\mu \times 1_{S_\ell}\times \sgn_{S_{n-\ell}} 
\Ind_{S_m \times S_\ell \times S_{n-\ell}}^{S_{m+n}} \rangle > 0 \]
then $[\lambda/\mu]$ has no four boxes making the shape $(2,2)$. 
\end{lemma}

\begin{proof}
By the versions of Pieri's rule and Young's rule proved at the end of~\S\ref{sec: Pieri},
the hypothesis implies that $\lambda$ is obtained from $\mu$ by adding a horizontal
strip of size $\ell$ and then a vertical strip of size $n-\ell$. If two boxes
from a horizontal strip are added to row $i$ then at most one box can be added
below them in row $i+1$ by a vertical strip. The result follows.
\end{proof}

\begin{lemma}\label{lemma: hookCase}
If $\lambda$ is a partition of $n$ and $\rho$ is an $n$-cycle then $\chi^\lambda(\rho) \not=0$
if and only if $\lambda = (n-\ell,1^\ell)$ where $0 \le \ell < n$.
\end{lemma}

\begin{proof}
Write $\mathrm{Cent}_{S_n}(\rho)$ for the centraliser subgroup of $\rho$ in $S_n.$
By a column orthogonality relation (see \cite[(31.13)]{CR})
\[ \sum_\lambda \chi^\lambda (\rho)^2 = |\mathrm{Cent}_{S_n}(\rho)| = n, \]
and the sum is over all partitions $\lambda$ of $n$. 
By~\eqref{eq: skewCase} in the case proved in~\textsection \ref{sec: snakeCase},
we have $\chi^{(n-\ell,1^\ell)}(\rho) = (-1)^{\ell-1}$ for $0 \le \ell < n$.
Therefore the partitions $(n-\ell,1^\ell)$ give all the non-zero summands.
\end{proof}

\begin{proposition}
Let $\lambda/\mu$ be a skew partition of size $n$ and let $\rho \in S_n$
be an $n$-cycle. If $\lambda / \mu$ is not a border strip
then $\chi^{\lambda/\mu}(\rho) = 0$.
\end{proposition}

\begin{proof}
If $[\lambda/\mu]$ is disconnected then it is clear
from the Standard Basis Theorem (Theorem~\ref{thm: SBT}(ii)) that 
$S^{\lambda / \mu}$ is isomorphic to a module induced from a proper Young subgroup $S_{n-\ell}
\times S_\ell$ of $S_n$. Since no conjugate of $\rho$ lies in this subgroup, we have $\chi^{\lambda/\mu}(\rho) = 0$.

In the remaining case $[\lambda/\mu]$ has four boxes making the shape $(2,2)$.
By either Pieri's rule or Young's rule, we have 
\[\langle 1_{S_\ell} \times \sgn_{S_{n-\ell}}
\ind_{S_\ell \times S_{n-\ell}}^{S_n}, \chi^{(n-\ell,1^{\ell})} \rangle = 1.\]
Hence
\begin{align*} \langle \chi^\lambda, \chi^\mu \times 1_{S_\ell} \times \sgn_{S_{n-\ell}} \Ind_{S_m \times S_\ell
\times S_{n-\ell}}^{S_{m+n}} \rangle 
&\ge \langle 
\chi^\lambda, \chi^\mu \times \chi^{(n-\ell,1^{\ell})} \Ind_{S_m \times S_n}^{S_{m+n}} \rangle \\
&= \langle \chi^{\lambda/\mu}, \chi^{(n-\ell,1^\ell)} \rangle \end{align*}
where the equality follows from Lemma~\ref{lemma: JP}.
By Lemma~\ref{lemma: YP} the left-hand side is $0$. It follows that
$\langle \chi^{\lambda/\mu}, \chi^{(n-\ell,1^\ell)} \rangle = 0$ for $0 \le \ell < n$.
By Lemma~\ref{lemma: hookCase}, this implies the result.
\end{proof}
\section*{Acknowledgements}
The authors thank two anonymous referees for their careful reading of an earlier version of this paper. 

\bibliographystyle{amsplain}
\bibliography{bibliography}

\end{document}